\documentclass{amsart}
\usepackage[utf8x]{inputenc}
\usepackage{amsmath,amsthm,amsfonts,amssymb,latexsym, enumerate}

\begin{document}

\theoremstyle{plain}
\newtheorem{thm}{Theorem}[section]
\newtheorem{lem}[thm]{Lemma}
\newtheorem{pro}[thm]{Proposition}
\newtheorem{cor}[thm]{Corollary}
\newtheorem{statement}[thm]{}

\theoremstyle{definition}
\newtheorem{que}[thm]{Question}
\newtheorem{rem}[thm]{Remark}
\newtheorem{defi}[thm]{Definition}
\newtheorem{Question}[thm]{Question}
\newtheorem{Conjecture}[thm]{Conjecture}
\newtheorem{exa}[thm]{Example}

\def\CO{{\mathcal{O}}}
\def\CR{{\mathcal{R}}}
\def\OD{{\mathcal{O}D}}
\def\OG{{\mathcal{O}G}}
\def\OGb{{\mathcal{O}Gb}}
\def\OH{{\mathcal{O}H}}
\def\ON{{\mathcal{O}N}}
\def\OP{{\mathcal{O}P}}
\def\OQ{{\mathcal{O}Q}}
\def\OR{{\mathcal{O}R}}

\def\Br{\mathrm{Br}} 
\def\codim{\mathrm{codim}}
\def\End{\mathrm{End}}
\def\Hom{\mathrm{Hom}}
\def\IBr{\mathrm{IBr}} 
\def\Ind{\mathrm{Ind}} 
\def\Irr{\mathrm{Irr}} 
\def\id{\mathrm{id}} 
\def\Ker{\mathrm{Ker}} 
\def\mod{\mathrm{mod}} 
\def\Res{\mathrm{Res}} 
\def\op{\mathrm{op}}
\def\rk{\mathrm{rk}} 
\def\Syl{\mathrm{Syl}} 
\def\Tr{\mathrm{Tr}} 

\def\ten{\otimes}
\def\tenO{\otimes_{\mathcal{O}}}
\def\tenK{\otimes_K}
\def\tenk{\otimes_k}

\title[Anchors]
{Anchors of irreducible characters}

\author[Kessar]{Radha Kessar}
\address{Department of Mathematics, City University London EC1V 0HB,  
United Kingdom}
\email{radha.kessar.1@city.ac.uk}

\author[K\"ulshammer]{Burkhard K\"ulshammer}
\address{Institut f\"ur Mathematik, Friedrich-Schiller-Universit\"at, 
07743 Jena, Germany}
\email{kuelshammer@uni-jena.de}

\author[Linckelmann]{Markus Linckelmann}
\address{Department of Mathematics, City University London EC1V 0HB,  
United Kingdom}
\email{markus.linckelmann@city.ac.uk}

\thanks{The second author gratefully acknowledges support by the DFG 
(SPP 1388).}

\begin{abstract}
Given a prime number $p$, every irreducible character $\chi$ of 
a finite group $G$ determines a unique conjugacy class of 
$p$-subgroups of $G$ which we will call the {\it anchors} of $\chi$.  
This invariant has been considered by Barker in the context of
finite $p$-solvable groups.
Besides proving some basic properties of these anchors, we 
investigate the relation to other $p$-groups which can be attached to 
irreducible characters, such as defect groups, vertices in the sense 
of J.~A.~Green and vertices in the sense of G.~Navarro.
\end{abstract}

\maketitle

\begin{center}
\textit{Dedicated to the memory of J.~A.~Green}
\end{center}

\bigskip

\begin{center}
\today
\end{center}

\section{Introduction}

Let $p$ be a prime number and $\CO$ a complete discrete valuation
ring with residue field $k=$ $\CO/J(\CO)$ of characteristic $p$ and
field of fractions $K$ of characteristic $0$.
For $G$ a finite group, we denote by $\Irr(G)$ the set of characters of 
the simple $KG$-modules. For $\chi \in \Irr(G)$, we denote by
$e_\chi$ the unique primitive idempotent in $Z(KG)$ satisfying
$\chi(e_\chi)\neq$ $0$. 
The $\CO$-order $\OG e_\chi$ in the simple $K$-algebra $KG e_\chi$ is 
a $G$-interior $\CO$-algebra, via the group homomorphism 
$G \to$ $(\OG e_\chi)^\times$ sending $g\in$ $G$ to $ge_\chi$.
Since $(\OG e_\chi)^G = Z(\OG e_\chi)$ is a subring of the field
$Z(KGe_\chi)$, it follows that $\OG e_\chi$ is a primitive $G$-interior 
$\CO$-algebra. In particular, $\OG e_\chi$ is a primitive $G$-algebra. 
By the fundamental work of J.~A.~Green \cite{Gr}, it has a defect 
group. This is used in work of Barker \cite{Bar} to prove a part of a 
conjecture of Robinson (cf. \cite[4.1, 5.1]{RobLoc}) for blocks of 
finite $p$-solvable groups.
In order to distinguish this invariant from defect groups of blocks 
and from vertices of modules, we introduce the following terminology.

\begin{defi} \label{defn:anchors} 
Let $G$ be a finite group and let $\chi \in \Irr(G) $.  An
{\it anchor of $\chi$} is a defect group of the primitive $G$-interior 
$\CO$-algebra $\OG e_\chi$.
\end{defi} 

By the definition of defect groups, an anchor of an irreducible 
character $\chi$ of $G$ is a subgroup $P$ of $G$ which is minimal 
with respect to $e_\chi \in$ $(\mathcal{O}G e_\chi)_P^G$, where 
$(\OG e_\chi)_P^G$ denotes the image of the  relative trace map
$\Tr_P^G: (\OG e_\chi)^P \to (\OG e_\chi)^G$.
Green's general theory in \cite[\S 5]{Gr} implies that the anchors
of $\chi$ form a conjugacy class of $p$-subgroups of $G$.

For the remainder of the paper we make the blanket assumption that
$K$ and $k$ are splitting fields for the finite groups arising in
the statements below. In a few places, this assumption is not
necessary; see the Remark \ref{nonsplitRemark} below.

\begin{thm} \label{anchors-omnibus1} 
Let $G$ be a finite group and let $\chi \in \Irr (G)$. Let $B$ be the  
block of $\CO G$ containing $\chi $ and let $L$ be an $\OG$-lattice 
affording $\chi$. Let $P$ be an anchor of $\chi$  and denote by  
$\Delta P$ the image $\{ (x, x) \,  |  \,   x \in P \}$ of $P$ 
under  the diagonal embedding of $G$ in $G \times G$. 
The following hold.

\begin{enumerate}[\rm (a)]

\item $P$ is contained in a defect group of $B$.

\item  $P$ contains a vertex of $L$.

\item  We have $O_p(G) \leq P$.    

\item  
The  suborder $\OP e_{\chi}$ of $\OG e_{\chi}$ is local, and  
$\OG e_{\chi}$ is a separable extension of $\OP e_{\chi}$.

\item $\Delta P $ is contained in a vertex of the 
$\CO(G \times G)$-module $ \OG e_{\chi} $  and  $P\times P$ 
contains a vertex of $ \OG e_{\chi} $. Moreover, $\Delta P$ is a 
vertex of $\OG e_{\chi} $ if and only if $\chi$ is of defect zero. 

\end{enumerate}
\end{thm}

For $G$ a finite group, we denote by $\IBr(G)$ the set of
$\CO$-valued Brauer characters of the simple $kG$-modules.
We denote by $G^\circ$ the set of $p'$-elements in $G$, and for
$\chi$ a $K$-valued class function on $G$, we denote by
$\chi^\circ$ the restriction of $\chi$ to $G^\circ$.

\begin{thm} \label{anchors-omnibus2} 
Let $G$ be a finite group and $\chi \in \Irr (G)$. Let $B$ be the  
block of $\OG$ containing $\chi$ and let $L$ be an $\OG$-lattice 
affording $\chi$. Let $P$ be an anchor of $\chi$. The following hold.

\begin{enumerate}[\rm (a)]

\item 
If $\chi^\circ\in$ $\IBr(G)$, then $L$ is unique up to isomorphism, $P$ 
is a vertex of $L$, and $P\times P$ is a vertex of the
$\CO(G\times G)$-module $\OG e_\chi$.

\item \label{omni-mult} 
Let $\tau $  be  a local point of $P$ on $\CO Ge_{\chi} $. Then the 
multiplicity module of $\tau$ is simple. In particular,   
$O_p(N_G(P_{\tau})) =P$ and $P$ is centric in a fusion system of $B$.

\item  
If $B$ has an abelian defect group $D$, then $D$ is an anchor of $\chi$.

\item 
If $\chi$ has height zero, then $P$ is a defect group of $B$, and $P \times P$
is a vertex of the $\CO(G \times G)$-module $\CO Ge_{\chi}$.
\end{enumerate}
\end{thm}

The hypothesis $\chi^\circ\in$ $\IBr(G)$ in the first statement of
Theorem \ref{anchors-omnibus2} holds if $\chi$ is a height zero
character of a nilpotent block. If $G$ is $p$-solvable, then
by the Fong-Swan theorem \cite[\S 22]{CR}, for any $\varphi\in$ 
$\IBr(G)$ there is $\chi\in$ $\Irr(G)$ such that $\chi^\circ=$ 
$\varphi$. The fact that anchors are centric is essentially proved in 
the proof of \cite[Theorem]{Bar} as an immediate consequence of results 
of Kn\"orr \cite{K}, Picaronny-Puig \cite{PiPu}, and Th\'evenaz 
\cite{T-DG}; see the proof of \ref{multpro} below for details.

The next result shows that anchors are invariant under Morita 
equivalences given by a bimodule with endopermutation source, hence in 
particular under source algebra equivalences.

\begin{thm} \label{omni-source} 
Let $G$, $G'$ be finite groups. Let $B$, $B'$ be blocks of 
$\OG$, $\OG'$, with defect groups $D$, $D'$, respectively.
Suppose that $B$ and $B'$ are Morita equivalent via a 
$B$-$B'$-bimodule $M$ which has an endopermutation source, when 
viewed as an $\CO(G\times G')$-module. 
Let $\chi\in$ $\Irr(B)$ and $\chi'\in$ $\Irr(B')$ such that
$\chi$ and $\chi'$ correspond to each other under the Morita
equivalence determined by $M$. Then there is an isomorphism
$D\cong$ $D'$ sending an anchor of $\chi$ to an anchor of $\chi'$.
In particular, if $B$ and $B'$ are source algebra equivalent,
then $\chi$ and $\chi'$ have isomorphic anchors.
\end{thm}

The last statement in Theorem \ref{omni-source} can be made more
precise: if $B$, $B'$ are source algebra equivalent, then the
isomorphism $D\cong$ $D'$ can be chosen to have an extension to
a source algebra isomorphism; see Theorem \ref{sourcealgebra} below.

In \cite{N}  Navarro associated, via the theory of special characters,  
to each ordinary irreducible character $\chi$  of a $p$-solvable group 
$G$, a $G$-conjugacy class of pairs $(Q, \delta)$, where $Q$ is a  
$p$-subgroup of $G$ and $\delta $ is an ordinary irreducible character 
of $Q$,  which behave in certain ways as the  Green vertices of 
indecomposable modules (see also \cite{Ea},\cite{C} and \cite{CL}). 
We call such a pair $(Q, \chi)$  a {\it Navarro vertex} of $\chi$. We 
prove the following two results relating Navarro vertices and anchors 
(see Section \ref{normalSection} for the definitions).

\begin{thm} \label{contain}    
Let $G$ be a finite $p$-solvable group and let $\chi \in \Irr(G)$ such
that $\chi^\circ\in$ $\IBr(G)$. Let $(Q,\delta)$ be a  
Navarro vertex of $\chi$. Then $Q$ contains an anchor of $\chi$. 
Further, if $p$ is odd or $\delta$ is the trivial character, then 
$Q$ is an anchor of $\chi$.
\end{thm}  

\begin{thm} \label{iscontained} 
Let $G$ be a finite group of odd order, and let $\chi \in$ $\Irr(G)$ 
have Navarro vertex $(Q, \delta)$. Then $Q$ is contained in an anchor 
of $G$. 
\end{thm}

We give examples which show that equality does not always hold in the 
above theorem. We also give examples which show that if $|G|$ is 
even, then a Navarro vertex need not be contained in any anchor of 
$\chi $. 

Section \ref{OrderSection} of the paper contains  some basic
properties of quotient orders of finite group algebras. In 
Section \ref{AnchorSection}, we prove the theorems~\ref{anchors-omnibus1}
and \ref{anchors-omnibus2}. In section \ref{sourceSection} we
characterise anchors at the source algebra level, and use this
to prove Theorem \ref{omni-source}. 
In Section \ref{normalSection} we prove some properties of anchors  
with respect to normal subgroups.
Section \ref{NavarroSection} contains the proofs of   
Theorems ~\ref{contain} and \ref{iscontained}.     
In Section \ref{LiftingSection} we compare anchors of $\OG e_\chi$
to the defect groups of $k\tenO \OG e_\chi$.

\begin{rem} \label{nonsplitRemark}
The splitting field hypothesis on $K$ and $k$ is not needed in
Theorem \ref{anchors-omnibus1} and the Propositions
\ref{anchorscontainvertices}, \ref{bigradical}, \ref{separable},
and \ref{bimod}, on which the proof of Theorem \ref{anchors-omnibus1}
is based. This hypothesis is also not needed in Theorem
\ref{sourcealgebra}, stating that anchors can be read off the
source algebras of a block.
\end{rem}

\begin{rem}
Let $G$ be a finite group, $\chi\in$ $\Irr(G)$, and $b$ the block
idempotent of the block of $\OG$ to which $\chi$ belongs; that is,
$b$ is the primitive idempotent in $Z(\OG)$ satisfying $be_\chi=$ 
$e_\chi$. Let $P$ be an anchor of $\chi$; that is, $P$ is a minimal 
subgroup of $G$ such that there exists an element $c\in$ 
$(\OG e_\chi)^P$ satisfying $e_\chi=$ $\Tr^G_P(c)$. We clearly have 
$(\OG)^Pe_\chi\subseteq$ $(\OG e_\chi)^P$, but this inclusion need not
be an equality, and this is one of the main issues for calculating
anchors. The spaces $(\OG)^Pe_\chi$ and $(\OG e_\chi)^P$ have the same 
$\CO$-rank, since $(KG)^Pe_\chi=$ $(KG e_\chi)^P$. An equality
$(\OG)^Pe_\chi=$ $(\OG e_\chi)^P$ implies that $P$ is a defect
group of the block of $\OG$ to which $\chi$ belongs; see
Proposition \ref{Pfixedpoints} below.
\end{rem}

\section{Orders of characters} \label{OrderSection}

Let $G$ be a finite group and $\chi \in \Irr(G)$. 
Then $\OG e_\chi$ is an $\CO$-order in the simple $K$-algebra 
$KGe_\chi$, called the $\CO$-\emph{order} of $\chi$. 
In general, $\OG e_\chi$ is not a subalgebra of $\OG$, but it is
an $\CO$-free quotient of $\OG$. The map $f: \OG \to \OG e_\chi$ 
sending $x\in$ $\OG$ to $xe_\chi$ is an epimorphism of $\CO$-orders; 
in particular, the $\CO$-orders $\OG e_\chi$ and $\OG/\Ker(f)$ are 
isomorphic. Since we assume that $K$ and $k$ are splitting fields
for all finite groups arising in this paper, we have
$e_\chi =$ $\frac{\chi(1)}{|G|} \sum_{g \in G} \chi(g^{-1})g$,
and as a $K$-algebra, $KGe_\chi$ is isomorphic to the matrix algebra
$K^{\chi(1) \times \chi(1)}$. In particular, $KGe_\chi$ has
up to isomorphism a unique simple left module $M_\chi$, and
we have $\dim_K(M_\chi) = \chi(1)$. Since the $\CO$-rank of
$\OG e_\chi$ is $\chi(1)^2$, it follows that
$k \tenO \OG e_\chi$ is a $k$-algebra of dimension $\chi(1)^2$ 
whose isomorphism classes of simple modules are in bijection with 
the set
$$\IBr(G|\chi) = \{ \varphi \in \IBr(G): d_{\chi,\varphi} \neq 0 \}$$
where $d_{\chi,\varphi}$ denotes the decomposition number attached to 
$\chi$ and $\varphi$, defined by the equation $\chi^\circ=$
$\sum_{\varphi\in\IBr(G)}\ d_{\chi,\varphi} \varphi$. 
It is well-known that $d_{\chi,\varphi}=$ 
$\chi(i)$, where $i$ is a primitive idempotent in $\OG$ such that 
$\OG i$ is a projective cover of a simple module with Brauer character 
$\varphi$. In particular, $d_{\chi,\varphi}\neq$ $0$ if and only
if $ie_\chi\neq$ $0$. The simple $k \tenO \OG e_\chi$-module 
$N_\varphi$ corresponding to $\varphi \in \IBr(G|\chi)$ has dimension 
$\varphi(1)$. Thus we have an isomorphism of $k$-algebras
$$k \tenO \OG e_\chi / J(k \tenO \OG e_\chi) \cong 
\prod_{\varphi \in \IBr(G|\chi)} k^{\varphi(1) \times \varphi(1)}.$$
Let $P(N_\varphi)$ denote an $\OG e_\chi$-lattice which is a 
projective cover of $N_\varphi$. Then the $KGe_\chi$-module 
$K\tenO P(N_\varphi)$ is isomorphic to $M^{d_{\chi,\varphi}}$; in 
particular, we have  $\rk_\CO(P(N_\varphi)) = d_{\chi,\varphi} \chi(1)$.
Setting $\ell_\chi = |\IBr(G|\chi)|$, 
the decomposition matrix of the $\CO$-order $\OG e_\chi$ is the 
$1 \times \ell_\chi$-matrix
$$\Delta_\chi = (d_{\chi,\varphi}: \varphi \in \IBr(G|\chi))\ .$$
Hence the Cartan matrix of $\OG e_\chi$ is the 
$\ell_\chi \times \ell_\chi$-matrix
$$C_\chi = \Delta_\chi^\top \Delta_\chi = 
(d_{\chi,\varphi}d_{\chi,\psi}: \varphi,\psi \in \IBr(G|\chi)).$$
By definition, $C_\chi$ has rank 1. The only non-zero invariant factor 
of $C_\chi$ is
$$\gcd (d_{\chi,\varphi}d_{\chi,\psi}: \varphi,\psi \in \IBr(G|\chi)) = 
\gcd(d_{\chi,\varphi}: \varphi\in \IBr(G|\chi))^2.$$
If $\chi^\circ \in \IBr(G)$, then it is well-known  that the 
$\CO$-order $\OG e_\chi$ is isomorphic to 
$\CO^{\chi(1) \times \chi(1)}$ (see e. g.  \cite[Prop. 4.1]{KKL}), 
and thus the $k$-algebra $k \tenO \OG e_\chi$ is isomorphic to 
$k^{\chi(1) \times \chi(1)}$. In this case the decomposition matrix 
$\Delta_\chi$ is the $1 \times 1$-matrix $(1)$, and so is the Cartan 
matrix $C_\chi$. 

Since $Z(KGe_\chi) = Ke_\chi \cong K$, we have $Z(\OG e_\chi) =$
$\CO e_\chi \cong \CO$; in particular, $Z(\OG e_\chi)$ is a local 
$\CO$-order, and hence $Z(k \tenO \OG e_\chi)$ is 
a local $k$-algebra, by standard lifting theorems for central 
idempotents. It is obvious that
$$Z(k \tenO \OG e_\chi) \supseteq k \tenO Z(\OG e_\chi) = 
k \otimes e_\chi.$$
This inclusion can be proper, or equivalently, the canonical map
$Z(\OG e_\chi)\to$ $Z(k\tenO \OG e_\chi)$ need not be surjective.
The following example illustrates this.

\begin{exa} 
Let $G$ be the dihedral group of order 8, let $\chi \in \Irr(G)$ with 
$\chi(1) = 2$, and let $p=2$. We represent $G$ in the form 
$G = \langle a,b \rangle$ where
$$a = \left( \begin{array}{cc}
0 & 1 \cr -1 & 0
\end{array} \right) \quad \mathrm{and} \quad 
b = \left( \begin{array}{cc}
0 & 1 \cr 1 & 0 \end{array} \right). $$
Then $\mathcal{O} G e_\chi$ is isomorphic to the subalgebra $\Lambda$ 
of $\CO^{2 \times 2}$ generated by $a$ and $b$. 
Note that $ba = -ab$, so that $(1 \otimes b)(1 \otimes a) = 
(1 \otimes a)(1 \otimes b)$ in $k \otimes_\mathcal{O} \Lambda$. 
This shows that $k \tenO \OG e_\chi \cong k \tenO \Lambda$ is 
commutative and of dimension 4.
\end{exa}

\section{Proofs of the theorems \ref{anchors-omnibus1} and
\ref{anchors-omnibus2}} \label{AnchorSection}

\begin{pro} \label{anchorscontainvertices}
Let $G$ be a finite group, $\chi \in \Irr(G)$, and let $P$ be an
anchor of $\chi$. Let $B$ be the block of $\OG$ containing $\chi$, 
and let $L$ be an $\OG$-lattice affording $\chi$. The following hold.

(i) $P$ is contained in a defect group of $B$.

(ii) $P$ contains a vertex of $L$.
\end{pro}

\begin{proof}
Let $D$ be a defect group of $B$. Then there exists an element 
$x \in (\OG)^D$ such that $\Tr_D^G(x) = 1_B$.
Then $xe_\chi \in (\OG e_\chi)^D$, and $\Tr_D^G(xe_\chi) = 
\Tr_D^G(x)e_\chi = 1_Be_\chi = e_\chi$. Thus $D$
contains an anchor of $\chi$, and (i) follows.
Let $y \in (\OG e_\chi)^P$ such that $\Tr_P^G(y) = e_\chi$. 
Then the map $\eta: L \to L$ sending $z$ to $yz$, is an 
element in $\End_{\OP}(L)$ such that $\Tr_P^G(\eta) = \id_L$. By 
Higman's criterion, $P$ contains a vertex of $L$, whence the
result.
\end{proof}

\begin{pro}\label{bigradical}
Let $G$ be a finite group, $\chi\in$ $\Irr(G)$, and let $P$
be an anchor of $\chi$. Then $O_p(G) \leq P$.
\end{pro}

\begin{proof}
Set $N=$ $O_p(G)$. Arguing by contradiction, suppose that $P$ does
not contain $N$. Then $P$ is a proper subgroup of $PN$.
For $g \in N$, we have $g-1 \in J(\ON) \subseteq J(\OG)$. It
follows that $gd-d$, $dg^{-1}-d$, and $gdg^{-1}-d=$ $gdg^{-1}-dg^{-1}+
dg^{-1}-d$ are contained in $J(\OG e_\chi)$ for all $g\in$ $N$ and
all $d\in$ $\OG e_\chi$.  Let $d\in$ $(\OG e_\chi)^P$ such that
$\Tr^G_P(d)=$ $e_\chi$. By the above , we have
$\Tr^{PN}_P(d)-|PN:P|d\in$ $J(\OG e_\chi)$. Since $p$ divides 
$|PN:P|$, it follows that $x=\Tr^{PN}_P(d)\in$ $J(\OG e_\chi)$.
Applying $\Tr^G_{PN}$ to this element shows that
$e_\chi=$ $\Tr^G_P(d)=$ $\Tr^G_{PN}(x)\in$ $J(\OG e_\chi)$, a
contradiction.
\end{proof}

Let $R$ be an $\CO$-order with unitary suborder $S$. We recall that 
$R$ is called a \emph{separable extension} of $S$ if the 
multiplication map $\mu: R \otimes_S R \to R$ sending 
$x \otimes y$ to $xy$ for all $x$, $y\in$ $R$
splits as a map of $R$-$R$-bimodules. This is equivalent to the 
condition that $1_R = \mu(z)$ for some $z \in (R \otimes_S R)^R$.

\begin{pro}\label{separable}
Let $G$ be a finite group and $\chi \in \Irr(G)$. Let $P$ be an anchor 
of $\chi$. Then the $\CO$-order $\OG e_\chi$ is a separable extension 
of its local suborder $\OP e_\chi$.
\end{pro}

\begin{proof}
The $\CO$-order $\OP$ is local, and the map 
$\OP \to \OG e_\chi$ induced by multiplication with $e_\chi$
is a homomorphism of $\CO$-algebras with image $\OP e_\chi$. Thus 
$\OP e_\chi$ is a local suborder of $\OG e_\chi$. Let $T$ be a 
transversal for $G/P$, and let $d \in (\OG e_\chi)^P$ such that
$$e_\chi = \Tr_P^G(d) = \sum_{g \in T} gdg^{-1}.$$
Then the element $x = \sum_{g \in T} gde_\chi \otimes g^{-1}e_\chi \in 
\OG e_\chi \otimes_{\OP e_\chi}
\OG e_\chi$ is independent of the choice of $T$ since
$$gude_\chi \otimes u^{-1}g^{-1} e_\chi = gde_\chi ue_\chi \otimes 
u^{-1}e_\chi g^{-1}e_\chi = gde_\chi \otimes g^{-1}e_\chi$$
for $g \in T$ and $u \in P$. But, for $h \in G$, $hT$ is another 
transversal for $G/P$. Thus
$$x = \sum_{g \in T} hgde_\chi \otimes g^{-1}h^{-1} e_\chi = hxh^{-1}.$$
This shows that $x \in (\OG e_\chi \otimes_{\OP e_\chi} 
\OG e_\chi)^{\mathcal{O}Ge_\chi}$, and
$$\mu(x) = \sum_{g \in T} gde_\chi g^{-1}e_\chi = 
\Tr_P^G(d)e_\chi^2 = e_\chi^3 = e_\chi$$
where $\mu: \OG e_\chi \otimes_{\OP e_\chi} \OG e_\chi 
\to\OG e_\chi$, denotes the multiplication map sending 
$a\otimes b$ to $ab$. The result  follows.
\end{proof}

\begin{rem}
The proposition above implies that $k \tenO \OG e_\chi$ is also a 
separable extension of $k \tenO \OP e_\chi$. Note that 
$k \tenO \OP e_\chi$ is a homomorphic image of the group algebra 
$kP$. Thus, if $P$ is cyclic, then $k \tenO \OP e_\chi$ 
has finite representation type. 
\end{rem}

\begin{pro}\label{bimod}
Let $G$ be a finite group and $\chi \in \Irr(G)$. Let $P$ be an anchor 
of $\chi$. Then $\Delta P$ is  contained in a vertex of the 
$\CO(G \times G)$-module $\CO Ge_{\chi}$ and $P\times P $ contains   
a vertex of $\OG e_{\chi}$. Moreover, $\Delta P$ is a vertex of 
$\OG e_{\chi} $ if and only if $\chi$ is of defect zero. 
\end{pro}

\begin{proof}   
Viewing $\OG e_{\chi}$  as an $\CO(G  \times G)$-module, we have
$\OG e_{\chi}(\Delta P) \ne 0$, where $\OG e_\chi(\Delta P)$ is the Brauer
quotient of $\OG e_\chi$ (cf. \cite[\S 11]{T}).  
The first assertion follows by \cite[Exercise 27.2 (a)]{T}. 
Next, we claim that $\OG e_{\chi}$ is 
relatively $P\times G$-projective. Indeed, let $T$ be a transversal 
for $G/P$, and let $d \in (\OG e_\chi)^P$ such that
$$e_\chi = \Tr_P^G(d) = \sum_{g \in T} gdg^{-1}.$$
The map
$$ \OG e_{\chi} \to \CO(G\times G) \otimes_{\CO (P \times G)} 
\OG e_{\chi}, \ \ (x \to  \sum_{ u\in  T}     
(g,1) \otimes dg^{-1} x  ), \ \ x \in  \CO Ge_{\chi} $$
is an  $\CO(G \times G)$-module splitting of the  surjective  
$\CO(G \times G) $-module homomorphism 
$$\CO(G\times G) \otimes_{\CO (P \times G)}   \CO  G e_{\chi},  
\ \ (y \otimes y' \to yy'), \ \ y\in  \CO(G\times G)  , 
y' \in \CO G e_{\chi} $$  
proving the claim. Similarly, $\OG e_{\chi} $ is relatively 
$G\times  P$-projective.   Let $R_1 $ be a vertex of 
$\OG e_{\chi}$ contained in  $G \times P$ and let $R_2$ be a vertex 
of $\CO G e_{\chi} $  contained in $P\times G $. Since $ R_1 $ and 
$R_2$ are $G \times G$-conjugate, it follows that 
$R_1 \leq  \,^x P \times P$ for some $x \in G$ and hence that  
$\OG e_{\chi} $   is relatively  $P \times P = 
\,^{(x^{-1}, 1) } (\,^xP \times P )$-projective.  This proves the 
second assertion.  

Finally, if $\Delta P$ is a vertex  of $\CO(G\times G)$, then  
$\Res^{\CO(G \times G) }_{\CO (G \times 1)} \CO G e_{\chi}$ is 
projective. In particular, the character of  $\CO G e_{\chi}$ as a 
left $\OG$-module vanishes on the $p$-singular elements of $G$. Since 
the character of $\OG e_{\chi}$ is a multiple of $\chi$, it follows 
that $\chi$ is of $p$-defect zero. Conversely, if $\chi $ is of 
$p$-defect zero, then by Proposition \ref{anchorscontainvertices} 
we have $P =1$, hence $1= P \times P = \Delta P $ is a vertex of 
$\OG  e_{\chi} $. 
\end{proof} 

\begin{proof}[{Proof of Theorem \ref{anchors-omnibus1}}]  
Proposition~\ref{anchorscontainvertices} implies 
(a) and (b) of  the theorem. Proposition~\ref {bigradical} proves (c). 
Part (d) follows from Proposition~\ref{separable} and  Part (e) is  
proved in Proposition~\ref{bimod}.    
\end{proof}

\begin{pro} \label{uniquelattice}
Let $G$ be a finite group, and let $\chi \in \Irr(G)$ with anchor $P$. 
Suppose that $\chi^\circ \in \IBr(G)$. Let $L$ be an $\OG$-lattice 
affording $\chi$. Then $L$ is unique up to isomorphism, $P$ is a 
vertex of $L$, and $P\times P$ is a vertex of the 
$\CO(G\times G)$-module $\OG e_\chi$.
\end{pro}

\begin{proof}
The hypotheses imply that the $\CO$-orders $\OG e_\chi$ and 
$\CO^{\chi(1) \times \chi(1)}$ are isomorphic. Since $\CO^{\chi(1)}$ 
is the only indecomposable $\CO^{\chi(1) \times \chi(1)}$-lattice, up 
to isomorphism, $L$ is the only indecomposable $\OG e_\chi$-lattice, 
up to isomorphism. Thus the canonical map 
$$\OG e_\chi \longrightarrow \End_\CO(L)$$ 
is an isomorphism of $G$-interior $\CO$-algebras, and hence these two 
primitive $G$-interior $\CO$-algebras have the same defect groups.
Higman's criterion implies that $P$ is a vertex of $L$. Moreover,
we have a canonical $\CO(G\times G)$-module isomorphism
$\End_\CO(L)\cong$ $L\tenO L^*$, where $L^*$ is the $\CO$-dual of $L$.
This implies that $P\times P$ is a vertex of $\OG e_\chi$. 
\end{proof}

We observe next that the multiplicity modules of the primitive
$G$-interior $\CO$-algebras $\OG e_\chi$ are simple. Background 
material on multiplicity modules can be found in 
\cite[\S 9 Appendix]{T-DG}.

\begin{pro} \label{multpro}
Let $G$ be a finite group and $\chi\in$ $\Irr(G)$. Let $P_\tau$
be a defect pointed group of the primitive $G$-interior $\CO$-algebra
$\OG e_\chi$. Then $P$ is an anchor of $\chi$ and the multiplicity 
module of $\tau$ is simple. In particular, we have $O_p(N_G(P_\tau))=$ 
$P$, and $P$ is centric in a fusion system of the block containing 
$\chi$.
\end{pro}

\begin{proof}
The fact that $P$ is an anchor of $\chi$ is a standard property of 
local pointed groups on primitive $G$-algebras (see e. g. 
\cite[(18.3)]{T}). 
As noted earlier, we have $(\OG e_\chi)^G\cong$ $\CO$. Set $\bar N=$ 
$N_G(P_\tau)/P)$. It follows from \cite[9.1.(c), 9.3.(b)]{T-DG} that 
the multiplicity module $V_\tau$ of $\tau$ is simple.
The well-known Lemma \ref{multOp} below implies that 
$O_p(N_G(P_\tau))=$ $P$.
By results of Kn\"orr in \cite{K}, every vertex of a lattice with 
irreducible character $\chi$ is centric in a fusion system of the block 
containing $\chi$. Centric subgroups in a fusion system are upwardly 
closed. Since the anchor $P$ of $\chi$ contains a vertex of every 
lattice with character $\chi$, the last statement follows. 
One can prove this also by applying the results of \cite{PiPu}
directly to the $G$-interior $\CO$-algebra $\OG e_\chi$.
\end{proof}

\begin{lem} \label{multOp}
Let $G$ be a finite group, $A$ a primitive $G$-algebra, and let
$P_\tau$ be a defect pointed group on $A$. If the multiplicity module 
of $\tau$ is simple, then $O_p(N_G(P_\tau))=$ $P$.
\end{lem}

\begin{proof}
Set $\bar N=$ $N_G(P_\tau)/P$. The multiplicity module $V_\tau$ of 
$\tau$ is a module over a twisted group algebra $k_\alpha\bar N$ for 
some $\alpha\in$ $H^2(\bar N;k^\times)$. Since $P_\tau$ is maximal, 
$V_\tau$ is projective (cf. \cite[9.1]{T-DG}). Since $V_\tau$ is also 
simple by the assumptions, it follows that $k_\alpha\bar N$ has a block 
which is a matrix algebra over $k$. By \cite[(10.5)]{T} we have
$k_\alpha\bar N\cong$ $kN'e$ for some central $p'$-extension $N'$
of $\bar N$ and some idempotent $e\in$ $Z(kN')$. Thus the multiplicity
module corresponds to a defect zero block of $kN'$.
Since $O_p(N')$ is contained in all defect groups of all blocks of 
$kN'$, it follows that $O_p(N')$ is trivial. By elementary group 
theory, the canonical map $N'\to$ $\bar N$ sends $O_p(N')$ onto 
$O_p(\bar N)$, and hence $O_p(\bar N)$ is trivial, or equivalently, 
$O_p(N_G(P_\tau))=$ $P$.  
\end{proof}

\begin{proof}[{Proof of Theorem \ref{anchors-omnibus2}}]  
Part (a) is proved in Proposition \ref{uniquelattice}.
Part (b) follows from  Proposition~\ref{multpro}, and  (c) is an
immediate consequence of (b). 
If $\chi$ has height zero then $\chi(1)_p = |G:D|_p$ where $D$ 
is a defect group of $B$.  
Then $D$ contains a vertex $Q$ of $L$, and $|G:Q|_p$ divides 
$(\rk_\CO L)_p = \chi(1)_p = |G:D|_p$. This implies $Q = D$. Hence $D$ 
is an anchor of $\chi$ in this case. A similar argument shows that 
$D \times D$ is a vertex of the $\CO(G \times G)$-module $\CO Ge_\chi$.
This proves (d). 
\end{proof}

\begin{pro} \label{Pfixedpoints}
Let $G$ be a finite group, $\chi\in$ $\Irr(G)$, and let $P$ be an
anchor of $\chi$. If $(\OG)^Pe_\chi=$ $(\OG e_\chi)^P$, then $P$
is a defect group of the block of $\OG$ to which $\chi$ belongs.
\end{pro}

\begin{proof}
Denote by $b$ the primitive idempotent in $Z(\OG)$ such that
$be_\chi=$ $e_\chi$. Suppose that $(\OG)^Pe_\chi=$ $(\OG e_\chi)^P$. 
Then $e_\chi=$ $\Tr^G_P(ce_\chi)=$ $\Tr^G_P(c)e_\chi$ for some $c\in$ 
$(\OGb)^P$. Thus $w=\Tr^G_P(c)$ is not contained in $J(Z(\OGb))$, 
hence invertible in $Z(\OGb)$. Therefore $b=$ $w^{-1}w=$ 
$\Tr^G_P(w^{-1}c)$, which implies that $P$ contains a defect group of 
$b$, hence is equal to a defect group of $b$ by Theorem 
\ref{anchors-omnibus1} (a).
\end{proof}

\section{Anchors and source algebras} \label{sourceSection}

We show that anchors of characters in a block can be read off the 
source algebras of that block, and use this to prove Theorem 
\ref{omni-source}. As before, we refer to \cite[\S 11]{T} for the
Brauer quotient and Brauer homomorphism.

Let $G$ be a finite group, $B$ a block of $\OG$, and $D$ a defect 
group of $B$. We denote by $\Irr(B)$ the subset of all $\chi\in$ 
$\Irr(G)$ satisfying $\chi(1_B)=$ $\chi(1)$. Let $i$ be a source 
idempotent in $B^D$; that is, $i$ is a primitive idempotent in
$B^D$ satisfying $\Br_D(i)\neq$ $0$. Then $A=$ $iBi=$ $i\OG i$ is a 
{\it source algebra of} $B$. We view $A$ as a $D$-interior $\CO$-algebra
with the embedding of $D\to$ $A^\times$ induced by multiplication
with $i$. By \cite[3.5]{Puigpointed}, the $A$-$B$-bimodule $iB=$ 
$i\OG$ and the $B$-$A$-bimodule $Bi=$ $\OG i$ induce a Morita 
equivalence between $A$ and $B$.
In particular, if $X$ is a simple $K\tenO B$-module, then
$iX$ is a simple $K\tenO A$-module, and this correspondence induces
a bijection between $\Irr(B)$ and the set of isomorphism classes
of simple $K\tenO A$-modules. Equivalently, the map $e_\chi\mapsto 
ie_\chi$ is a bijection between primitive idempotents in $Z(K\tenO B)$ 
and $Z(K\tenO A)$. If $U$ is a $B$-lattice with
character $\chi\in$ $\Irr(B)$, then $iU$ is an $A$-lattice such
that $K\tenO iU$ is a simple $K\tenO A$-module corresponding to $\chi$.
This Morita equivalence induces a bijection between the $\CO$-free
quotients of $B$ and of $A$. If $\chi\in$ $\Irr(B)$, then the
$\CO$-free quotient $\OG e_\chi=$ $Be_\chi$ corresponds to the 
$\CO$-free quotient $i\OG ie_\chi =$ $Aie_\chi=$ $Ae_\chi$. Note 
that $Ae_\chi$ is again a $D$-interior $\CO$-algebra, via the 
canonical surjection $A\to$ $Ae_\chi$. Note further that $Ae_\chi$ is 
a direct summand of $Be_\chi$ as an $\OD$-$\OD$-bimodule, since it is 
obtained from multiplying $Be_\chi$ on the left and on the right by 
the idempotent $ie_\chi$ in $(Be_\chi)^D$.
The next result shows that anchors of $\chi$ can be characterised in 
terms of the order $Ae_\chi$. This is based on a variation of 
standard arguments, similar to those used in \cite[\S 6]{Li}, 
identifying vertices of modules at the source algebra level.

\begin{thm} \label{sourcealgebra}
Let $G$ be a finite group, $B$ a block of $\OG$, $D$ a defect
group of $B$, and $i\in$ $B^D$ a source idempotent. Set $A=$ $iBi$.
Let $\chi\in$ $\Irr(B)$. 

\smallskip\noindent (i)
Let $Q$ be a $p$-subgroup of $G$ such that $(Be_\chi)(Q)\neq$ $0$.
Then there is $x\in$ $G$ such that ${^xQ}\leq$ $D$ and such
that $(Ae_\chi)({^xQ})\neq$ $\{0\}$.

\smallskip\noindent (ii)
Let $Q$ be a subgroup of $D$ of maximal order subject to 
$(Ae_\chi)(Q)\neq$ $\{0\}$. Then $Q$ is an anchor of $\chi$.
\end{thm}

\begin{proof}
We use basic properties of local pointed groups; see e. g. 
\cite[\S 18]{T} for an expository account of this material.
Let $\gamma$ be the local point of $D$ on $B$ containing $i$.
Let $Q$ be a $p$-subgroup of $G$ such that $(Be_\chi)(Q)\neq$ $\{0\}$.
Since $e_\chi$ is the unit element of $Be_\chi$, this
is equivalent to $\Br^{Be_\chi}_Q(e_\chi)\neq$ $0$. By considering
a primitive decomposition of $1_B$ in $B^Q$ it follows that there
is a primitive idempotent $j\in$ $B^Q$ such that 
$\Br_{Q}^{Be_{\chi}}(je_\chi)\neq$ $0$. Then necessarily also 
$\Br^B_Q(j)\neq$ $0$,
because the canonical map $B\to$ $Be_\chi$ sends $\ker(\Br_Q^B)$
to $\ker(\Br_Q^{Be_\chi})$. Thus $j$ belongs to a local point $\delta$ 
of $Q$ on $B$. Since the maximal local pointed groups on $B$ are
all $G$-conjugate, it follows that there is $x\in$ $G$ such that
${^xQ_\delta}\leq$ $D_\gamma$. In other words, after replacing
$Q_\delta$ by a suitable $G$-conjugate, we may assume that
$Q_\delta\leq$ $D_\gamma$, and hence that $j\in$ $A^Q$ for some
choice of $j$ in $\delta$. For this choice of $j$, we have 
$je_\chi\in$ $Ae_\chi$. Thus the condition 
$\Br^{Be_\chi}_Q(je_\chi)\neq$ $0$ is equivalent to 
$\Br^{Ae_\chi}_Q(je_\chi)\neq$ $0$; we use here the
fact, mentioned above, that $Ae_\chi$ is a direct summand of $Be_\chi$ 
as an $\OP$-$\OP$-bimodule. In particular, we have $(Ae_\chi)(Q)\neq$ 
$\{0\}$.  This proves (i). For (ii), let $Q$ be a subgroup of $D$ such 
that $(Ae_\chi)(Q)\neq$ $\{0\}$ and such that the order of $Q$ is 
maximal with respect to this property. Then $(Be_\chi)(Q)\neq$ 
$\{0\}$, and hence $Q$ is contained in an anchor $R$ of $\chi$. By 
(i), there is $x\in$ $G$ such that ${^xR}\leq$ $D$ and such that 
$(Ae_\chi)({^xR)}\neq$ $\{0\}$.
The maximality of $|Q|$ forces $Q=$ $R$, whence the result.
\end{proof}

Theorem \ref{sourcealgebra} implies that anchors are invariant under 
source algebra equivalences. By a result of Scott \cite{Scottnotes}
and Puig \cite[7.5.1]{Puigbook}, 
an isomorphism between source algebras is equivalent to a Morita
equivalence given by a bimodule with a trivial source (see also
\cite[\S 4]{Lisplendid} for an expository account). In order to 
extend the invariance of anchors to Morita equivalences given by
bimodules with endopermutation source, we need to describe these 
Morita equivalences at the source algebra level. Let $G$, $G'$ be finite 
groups, and let $B$, $B'$ be blocks of $\OG$, $\OG'$ with defect 
groups $D$, $D'$, respectively. By results of Puig in 
\cite[\S 7]{Puigbook}, 
a Morita equivalence between $B$ and $B'$ given by a bimodule with 
endopermutation source implies an identification $D=D'$ such that for 
some choice of source idempotents $i\in$ $B^D$, $i'\in$ $(B')^D$, 
setting $A=$ $iBi$ and $A'=$ $i'B'i'$, we have $D$-interior $\CO$-algebra 
isomorphisms
$$A'\cong e(S\tenO A)e\ ,\ \ \ A\cong e'(S^\op\tenO A')e'\ ,$$
where $S=$ $\End_\CO(V)$ for some indecomposable endopermutation
$\OD$-module $V$ with vertex $D$, and where $e$, $e'$ are primitive
idempotents in $(S\tenO A)^D$, $(S^\op\tenO A')^D$,
respectively, satisfying $\Br_D(e)\neq$ $0$, $\Br_D(e')\neq$ $0$.
These isomorphisms induce inverse equivalences between 
$\mod(A)$ and $\mod(A')$, sending an $A$-module $U$ to the
$A'$-module $e(V\tenO U)$, and an $A'$-module $U'$ to the
$A$-module $e'(V^*\tenO U')$. Here $V^*$ is the $\CO$-dual
of $V$; note that $\End_\CO(V^*)\cong$ $S^\op$ as $D$-interior
$\CO$-algebras. 

\begin{proof}[{Proof of Theorem \ref{omni-source}}]
We use the notation above.
Let $\chi\in$ $\Irr(B)$ and $\chi'\in$ $\Irr(B')$ such
that $\chi$ and $\chi'$ correspond to each other through the
Morita equivalence $\mod(B)\cong\mod(A)\cong\mod(A')\cong
\mod(B')$ described above. As mentioned at the beginning of this 
section, the primitive idempotent in $Z(K\tenO A)$ 
corresponding to $\chi$ is $ie_\chi$. Similarly, the primitive 
idempotent in $Z(K\tenO A')$ is equal to $i'e_{\chi'}$. The explicit
description of the Morita equivalence between $A$ and $A'$ above
implies that we have
$$i'e_{\chi'} = e\cdot(1_S\ten ie_\chi)$$
$$ie_{\chi} = e'\cdot(1_{S^\op}\ten i'e_{\chi'})$$
where these equalities are understood in the algebras $K\tenO A$ and
$K\tenO A'$. By Theorem \ref{sourcealgebra},
it suffices to show that for $Q$ a subgroup of $D$, we have 
$(Ae_\chi)(Q)\neq$  $\{0\}$ if and only if $(A'e_{\chi'})(Q)\neq$ 
$\{0\}$. It suffices to show one implication, because the other 
follows then from exchanging the roles of $A$ and $A'$. Thus it 
suffices to show that if $(Ae_\chi)(Q)=$ $\{0\}$,
then $(A'e_{\chi'})'(Q)=$ $\{0\}$. Let $Q$ be a subgroup of $D$ such 
that $(Ae_\chi)(Q)=$ $\{0\}$. We have $(S\tenO A)(1_S\ten ie_\chi)=$ 
$S\tenO Ae_\chi$. 
Since $S$ has a $D$-stable basis, it follows
from \cite[5.6]{Punil} that $(S\tenO Ae_\chi)(Q)=$ 
$S(Q)\tenk (Ae_\chi)(Q)=$ $\{0\}$.
Since $A'e_{\chi'}$ is obtained from $S\tenO Ae_\chi$ by left and 
right multiplication with the idempotent $e$, it follows that
$(A'e_{\chi'})(Q)=$ $\{0\}$ as required. 
\end{proof}

\section{Anchors and normal subgroups} \label{normalSection}

We prove in this section  some results on anchors of characters which
are induced from a normal subgroup  or inflated from quotients.
Since an anchor of an irreducible character $\chi$ contains a vertex 
of any lattice affording $\chi$, constructing suitable lattices
is one of the tools for getting lower bounds on anchors.
The following is well-known (we include a proof for the convenience
of the reader).

\begin{lem} \label{simpleheadLemma}
Let $G$ be a finite group and $\chi\in$ $\Irr(G)$. Let
$S$ be a simple $kG$-module with Brauer character $\varphi$
such that $d_{\chi \varphi}\neq$ $0$. Then there exists 
an $\OG$-lattice $L$ affording $\chi$ such that
$L$ has a unique maximal submodule $M$, and such that $L/M\cong$
$S$.
\end{lem}

\begin{proof}
Let $i$ be a primitive idempotent in $\OG$ such that $\OG i$ is
a projective cover of $S$. Since $\chi(i)=$ $d_{\chi \varphi}\neq$ $0$,
there is an $\CO$-pure submodule $L'$ of $\OG i$ such that 
$L=$ $\OG i/L'$ affords $\chi$. Since the projective indecomposable
$\OG$-module $\OG i$ has a unique maximal submodule and $S$ is 
its unique simple quotient, it follows that the image, denoted 
$M$, in $L$ of the unique maximal submodule of $\OG i$ is the
unique maximal submodule of $L$, and satisfies $L/M\cong S$.
\end{proof}

In \cite{P},  Plesken showed that if $G$ is a $p$-group and $\chi$ is 
an irreducible character  of $G$, then there exists  an  $\OG$-lattice 
affording $\chi$ whose vertex is $G$. Our next result is a slight 
variation on this theme.

\begin{pro} \label{GN1pro}
Let $G$ be a finite group, $N$ a normal subgroup of $G$ of 
$p$-power index, and $\chi\in$ $\Irr(G)$. If there exists $\varphi\in$ 
$\IBr(G)$ of degree not divisible by $|G:N|$ and such that 
$d_{\chi, \varphi}\neq$ $0$, then there exists an $\OG$-lattice $L$ 
with character $\chi$ which is not relatively $\ON$-projective.
In particular, in that case, $N$ does not contain the anchors
of $\chi$.
\end{pro}

\begin{proof}
Let $S$ be a simple $kG$-module with Brauer character $\varphi$ of
degree not divisible by $|G:N|$ and such that $d_{\chi, \varphi}\neq$ 
$0$.  By \ref{simpleheadLemma} there exists an $\OG$-lattice $L$
with a unique maximal submodule $M$ such that $\chi$ is the
character of $L$ and such that $L/M\cong$ $S$. Note that the
character of $M$ is also equal to $\chi$.
We will show that one of $M$ or $L$ is not relatively $\ON$-projective.
Arguing by contradiction, suppose that $L$ and $M$ are relatively
$\ON$-projective. By Green's indecomposability theorem, there
are indecomposable $\ON$-modules $Z$ and $U$ such that
$L\cong$ $\Ind^G_N(Z)$ and $M\cong$ $\Ind^G_N(U)$. 
Then $\chi=$ $\Ind^G_N(\tau)$, where $\tau$ is the character of
$Z$.  Since $\chi=$ $\Ind^G_N(\tau)$ is irreducible, it follows
that the different $G$-conjugates ${^x\tau}$ of $\tau$, with
$x$ running over a set of representatives $\CR$ of $G/N$ in $G$, are
pairwise different. Similarly, $\chi=$ $\Ind^G_N(\tau')$, where
$\tau'$ is the character of $U$. After replacing $U$ by ${^xU}$
for a suitable element $x\in$ $G$, we may assume that $\tau'=$ $\tau$.
By the above, we have $\Res^G_N(L)\cong$ $\bigoplus_{x\in\CR} {^xZ}$,
and the characters of these summands are the pairwise different
conjugates ${^x\tau}$ of $\tau$.
In particular, $\Res^G_N(L)$ has a unique $\CO$-pure summand with 
character $\tau$, and this summand is isomorphic to $Z$.  We denote 
this summand abusively again by $Z$. Similarly, $\Res^G_N(M)$ has a 
unique $\CO$-pure summand, abusively again denoted by $U$, with 
character $\tau$. Since $M\subseteq$ $L$ induces an equality
$K\tenO M=$ $K\tenO L$, it follows that $K\tenO U=$ $K\tenO Z$.
Moreover, we have $U\subseteq$ 
$K\tenO Z\cap L=$ $Z$, where the second equality holds as $Z$ is
$\CO$-pure in $L$. Thus the inclusion $M\subseteq$ $L$ is obtained
from inducing the inclusion map $U\subseteq$ $Z$ from $N$ to $G$. 
By the construction of $M$, the inclusion 
$M\subseteq$ $L$ induces a map $k\tenO M\to$ $k\tenO L$ with cokernel
$S$. Thus we have
$$\dim_k(S)=\codim(k\tenO M\to k\tenO L)=
|G:N|\codim(k\tenO U\to k\tenO Z)\ .$$
This contradicts the assumption that $\varphi(1)$ is not
divisible by $|G:N|$.
Thus one of $L$ or $M$ is not relatively $\ON$-projective.
\end{proof}

\begin{cor} \label{GN1cor}
Let $G$ be a finite group, $N$ a normal subgroup of $p$-power 
index, and let $\chi\in$ $\Irr(G)$. Let $P$ be an anchor of $\chi$.
If there exists $\varphi\in$ $\IBr(G)$
of degree prime to $p$ such that $d_{\chi \varphi}\neq$ $0$, then
$G=$ $PN$.
\end{cor}

\begin{proof}
Arguing by contradiction, suppose that $PN$ is a proper subgroup
of $G$. Since $G/N$ is a $p$-group, it follows that $PN$ is
contained in a normal subgroup $M$ of index $p$ in
$G$. Then $M$ contains every anchor of $\chi$, hence $M$ contains
the vertices of any $\OG$-lattice affording $\chi$. Let 
$\varphi\in$ $\IBr(G)$ such that $\varphi(1)$ is prime to $p$ and
such that $d_{\chi, \varphi}\neq$ $0$. Proposition
\ref{GN1pro} implies however that $|G:M|=$ $p$ divides $\varphi(1)$,
a contradiction.
\end{proof}

We record here an extension of ~\cite[Lemma 3]{P} which will be used 
in the next section.

\begin{pro}  \label{pro:ples}
Let $G$ be a finite group, $P$ a Sylow $p$-subgroup, and $\chi\in$
$\Irr(G)$. Suppose that $\Res^G_P(\chi)$ is irreducible  
and that there exists an irreducible Brauer character $\varphi$ of 
$p'$-degree of $G$ such that $d_{\chi,\varphi} \ne 0 $.  
Then there exists an $\OG$-lattice $L$ affording $\chi$ with vertex 
$P$. In particular, the Sylow $p$-subgroups of $G$ are the anchors 
of $\chi$.
\end{pro}

\begin{proof}    
Let $\pi $ be a generator of $J(\CO)$ and let $S$ be a simple 
$kG$-module with Brauer character  $\varphi $. By 
\ref{simpleheadLemma}, there is an $\OG$-lattice $L$ affording $\chi$ 
such that $k\tenO L$  has a simple head isomorphic to $S$. Let $N$ be 
the maximal submodule of $L$. Then $\pi L \subset N$  and $N$ is an 
$\OG$-lattice affording $\chi$. The invariant factors of the 
$\CO$-module $L/N $ are either $1$ or $\pi$ and the number of 
non-trivial invariant factors of $L/N$ equals $\dim_k(L/N)= 
\dim_k (S)$. Thus the product of the invariant factors of $L/N$ 
equals $\pi^{\dim_k S} $. By hypothesis, $\Res^G_P(L)$ is irreducible.
If $\Res^G_P(L)$ has vertex $P$, then $L$ has vertex $P$.  
So, we may assume that  $\Res^G_P(L)$ is relatively $U$-projective 
for some proper subgroup $U$ of $P$. By Green's indecomposability 
theorem, $\Res^G_P(L) = \Ind_U^P(M)$ for some $\CO U$-lattice $M$. 
By \cite[Lemma 3]{P}, $\Res^G_P(N)$ has vertex $P$, whence $N$ 
has vertex $P$.
Note that \cite[Lemma 3]{P} is stated for $\CO$ a localisation 
of  the $|P|$-th cyclotomic integers, but as remarked in 
\cite[Page 235]{P}, \cite[Lemma 3]{P} remains true in our setting. 
The second assertion of the proposition follows from 
Proposition~\ref{anchorscontainvertices}.
 \end{proof}

\begin{pro} \label{GN2pro}
Let $G$ be a finite group, $N$ a normal subgroup, and let 
$\chi\in$ $\Irr(G)$ such that $\chi=$ $\Ind^G_N(\tau)$
for some $\tau\in$ $\Irr(N)$. Let $V$ be an $\ON$-lattice with
character $\tau$. Suppose that the composition series of the
$kN$-modules $k\tenO {^xV}$, with $x$ running over a set of
representatives of $G/N$ in $G$, are pairwise disjoint.
Then $\OG e_\chi \cong \Ind^G_N(\ON e_\tau)$ as $G$-interior
$\CO$-algebras. In particular, $N$ contains the vertices of all
lattices affording $\chi$, and $N$ contains the anchors of $\chi$.
\end{pro}

\begin{proof}
It suffices to show that $e_\tau$ belongs to $\OG e_\chi$. Indeed, if 
this is true, then the assumptions on $\chi$ and $\tau$ imply that 
$e_\chi=$ $\Tr^G_N(e_\tau)$, and the different conjugates of $e_\tau$ 
appearing in $\Tr^G_N(e_\tau)$ are pairwise orthogonal idempotents in 
$\OG e_\chi$. In particular, we have $e_\tau\OG e_\tau=$ $\ON e_\tau$. 
It follows from  \cite[(16.6)]{T} that $\OG e_\chi\cong$ 
$\Ind^G_N(\ON e_\tau)$. It remains to show that $e_\tau$ belongs to 
$\OG e_\chi$. Let $I$ be a primitive decomposition of $1$ in $\ON$. 
Let $i\in$ $I$ such that $e_\tau i\neq$ $0$. This condition is 
equivalent to $k\tenO V$ having a composition factor isomorphic to the 
unique simple quotient $T_i$ of the $\ON$-module $\ON i$. Since the 
different $G$-conjugates of the $kN$-module $k\tenO V$ have pairwise 
disjoint composition series, it follows that $e_{^x\tau}i=$ $0$ for 
$x\in$ $G\setminus N$. Thus $e_\chi i=$ $e_\tau i\in$ $\OG e_\chi$ for 
any $i\in$ $I$ such that $e_\tau i\neq$ $0$. Taking the sum over all 
such $i$ implies that $e_\tau\in$ $\OG e_\chi$. The last statement 
follows from the fact that $e_\chi=$ $\Tr^G_N(e_\tau)$ and Higman's
criterion, for instance, or directly from the fact that $\Ind^G_N$
induces a Morita equivalence between $\ON e_\tau$ and $\OG e_\chi$.
\end{proof}

\begin{cor} \label{GN2cor}
Let $G$ be a finite group, $N$ a normal subgroup of $p$-power index, 
and let $\chi\in$ $\Irr(G)$ such that $\chi=$ $\Ind^G_N(\tau)$ for 
some $\tau\in$ $\Irr(N)$. Suppose that $d_{\chi, \varphi}$ is either 
$1$ or $0$ for every $\varphi\in$ $\IBr(G)$. Then $N$ contains the 
anchors of $\chi$.
\end{cor}

\begin{proof}
Let $V$ be an $\ON$-lattice affording $\tau$. Let $I$ be a primitive 
decomposition of $1$ in $\ON$. By Green's indecomposability theorem, 
$I$ remains a primitive decomposition in $\OG$. Let $i\in$ $I$.
We have
$$\chi(i) = \sum_x\ {^x\tau(i)}$$
where $x$ runs over a set of representatives $\CR$ of $G/N$ in $G$.
By the assumptions on the decomposition numbers of $\chi$,
the left side is either $1$ or $0$. Thus
either ${^x\tau(i)}=$ $0$ for all $x\in$ $\CR$, or there is
exactly one $x=x(i)\in$ $\CR$ with ${^x\tau(i)}\neq$ $0$.
This implies that the composition series of the different
$G$-conjugates of $k\tenO V$ are pairwise disjoint. The
result follows from \ref{GN2pro}.
\end{proof}

\begin{pro} \label{inflated}
Let $G$ be a finite group, $N$ a normal subgroup of $G$, and
$\chi\in$ $\Irr(G)$. Suppose that $\chi$ is the inflation to $G$ of 
an irreducible character $\psi\in$ $\Irr(G/N)$. 
Let $P$ be an anchor of $\chi$. Then $PN/N$ is an anchor 
of $\psi$, and $P \cap N$ is a Sylow $p$-subgroup of $N$. 
\end{pro}

\begin{proof}
Let $d\in$ $(\OG e_\chi)^P$ such that $\Tr^G_P(d)=$ $e_\chi$.
The assumptions imply that the canonical map $G\to$ $G/N$ induces
a $G$-algebra isomorphism $\OG e_\chi\cong \CO G/N e_\psi$
such that $N$ acts trivially on both algebras. Thus $e_\chi=$
$\Tr^G_P(d)=$ $|PN:P|\Tr^G_{PN}(d)$. This implies that $P$ is a
Sylow $p$-subgroup of $PN$, and hence that $P\cap N$ is a 
Sylow $p$-subgroup of $N$. Since the isomorphism
$\OG e_\chi\cong \CO G/N e_\psi$ sends $\Tr^G_{PN}(d)$ to 
$\Tr^{G/N}_{PN/N}(\bar d)$, where $\bar d$ is the canonical image
of $d$, it follows that $PN/N$ contains an anchor of $\psi$.
Using the fact that $P$ is a Sylow $p$-subgroup of $PN$, one
easily checks that any proper subgroup of $PN/N$ is of the form 
$QN/N$ for some proper subgroup $Q$ of $P$ containing $P\cap N$. 
The previous isomorphism implies that $PN/N$ is an anchor of $\psi$.
\end{proof}

\begin{exa}
(1) $p=2$, $G = \mathfrak{S}_3$, $\chi(1) = 2$: Then $\chi$ lies in 
a defect zero  block of $G$, hence by Theorem~\ref{anchors-omnibus1}, 
the trivial group is the only anchor of $\chi $. 

(2) $p=2$, $G = \mathfrak{S}_4$, $\chi(1) = 2$: Then $\chi$ is inflated 
from the character in part (1). In this case, by Proposition 
\ref{inflated}  the Klein four subgroup $V_4$ of $\mathfrak{S}_4$ is 
the only anchor of $\chi$. However, the defect groups of the block 
containing $\chi$ (i.e. of the principal block of $\mathfrak{S}_4$)
are the Sylow 2-subgroups of $\mathfrak{S}_4$ (i.e. dihedral groups 
of order 8). 

(3) $p=2$, $G = \mathfrak{S}_5$: All irreducible characters of $G$ 
except the one of degree 6 are of height zero in their block. So
their anchors coincide with their defect groups, by Theorem
\ref{anchors-omnibus2} (a) .\\
Now let $\chi \in \Irr(G)$ with $\chi(1) = 6$. Then $\chi$ is induced 
from an irreducible character of the alternating group 
$\mathfrak{A}_5$ of degree $3$. Thus there exists an 
$\OG$-lattice affording $\chi$ with vertex $V_4$. \\
On the other hand, $\chi$ is labelled by the partition 
$\lambda = (3,1,1)$ of $5$. By the remark on p. 511 of \cite{W}, the 
Specht module $S^\lambda$ is indecomposable, and 
$\mathfrak{S}_2 \times \mathfrak{S}_2$ is a vertex of $S^\lambda$, by 
Theorem 2 in \cite{W}. Thus the anchors of $\chi$ are Sylow 
$2$-subgroups of $G$, by Theorem \ref{anchors-omnibus1} (b).

\end{exa}

\section{Navarro vertices} \label{NavarroSection}

We prove  Theorems \ref{contain} and \ref{iscontained}.

\begin{thm}  
Let $G$ be a finite $p$-solvable group. Let $\chi \in \Irr(G)$ and let   
$(Q,\delta)$ be a Navarro vertex of $\chi$. 
Supppose that $ \chi^{\circ} \in \IBr(G)$. Then $Q$ contains an
anchor of $\chi$. Moreover, if $\delta= 1_Q$ or if $p$ is odd, 
then $Q$ is an anchor of $\chi $.
\end{thm} 

\begin{proof}   
Since  $\chi^\circ \in \IBr(G)$, there is a unique $\OG$-lattice $L$ 
affording $\chi$, up to isomorphism. Moreover, $k \tenO L$ is the 
unique simple $kG$-module with Brauer character $\chi^\circ$, up to  
isomorphism. Recall that there is a nucleus $(W,\gamma)$ of 
$\chi$ such that $\chi = \Ind_W^G(\gamma)$, and $Q \in 
\Syl_p(W)$ (cf. \cite[p. 2763]{N}).  Further, $\gamma \in
\Irr(W)$ has a unique factorization $\gamma = \alpha \beta$ where 
$\alpha \in \Irr(W)$ is $p'$-special and $\beta \in \Irr(W)$  is   
$p$-special. 
Going over to Brauer characters, we have $\chi^\circ = 
\Ind_W^G(\gamma^\circ)$ and $\gamma^\circ = \alpha^\circ \beta^\circ$; 
in particular, $\gamma^\circ, \alpha^\circ, \beta^\circ \in \IBr(W)$.
Let $ R$ be a vertex of the unique $\mathcal{O}W$-lattice affording 
$\gamma$ and let  $R_0$ be a vertex of the unique $kW$-module affording 
$\gamma^{\circ}$. Then, up to conjugation in $W$,  
$R_0 \leq R \leq Q$. Since $\chi = \Ind_W^G(\gamma)$, $R$  is  
also a vertex of the $\OG$-lattice affording $\chi$ and hence 
by  Proposition~\ref{anchorscontainvertices} (iii), $R$ is an anchor of 
$\chi $. This proves the first assertion.

Since $\alpha $ is $p'$-special,  the $p$-part of the degree of 
$\gamma $ equals the $p$-part of the degree of $\beta $. Since $G$ is 
$p$-solvable, it follows that   
$$ |R_0|  = \frac{|W|_p}{{\beta}(1)_p}. $$    

Now suppose that $p$ is odd. Since  $\beta^{\circ} $ is irreducible, 
by \cite[Lemma 2.1]{N11}, $\beta$ is linear. It follows from 
the above that $R_0=Q = R$, proving the second assertion when $p$ is 
odd. Since $\delta = \Res^W_Q (\beta)$, a similar argument works 
when $\delta=1_Q $. 
\end{proof} 

\begin{lem}  \label{ples-special}  
Let $G$ be a finite $p$-solvable group and $\chi\in\Irr(G)$. Suppose
that $\chi$ is $p$-special and that there exists $\varphi\in$ $\IBr(G)$ 
of $p'$-degree such that $d_{\chi\varphi} \ne 0$. Then there exists an 
$\OG$-lattice affording $\chi$ with vertex a Sylow $p$-subgroup of $G$.  
\end{lem}

\begin{proof} 
This is immediate from Proposition \ref{pro:ples} and the fact that the 
restriction of a $p$-special character  of $G$ to  a Sylow $p$-subgroup 
of $G$ is irreducible (cf. \cite[Prop.~6.1]{G}).
\end{proof}

The following is due to G.~Navarro.

\begin{lem} \label{odddecomposition} 
Let $G$ be a finite group of odd order and $\chi\in\Irr(G)$.
Suppose that $\chi$ is $p$-special. Then the trivial Brauer character 
of $G$ is a constituent of $\chi^{\circ} $.   
\end{lem}

\begin{proof}   
By the Feit-Thompson theorem, $G$ is solvable and hence $p$-solvable. 
Let $H$ be a $p$-complement of $G$ and let $\zeta $ be a primitive 
$|G|_p$-th root of unity. 
By \cite[Theorem~6.5]{G}, ${\mathbb Q}[\zeta]$  is a splitting field 
of $\chi$. Thus $\Res^G_H (\chi )$ is  a  rational valued  character 
of odd degree. Hence, $\Res^G_H (\chi)$ contains a real valued  
irreducible constituent, say $\alpha$. By Brauer's permutation lemma,   
the number of real-valued irreducible characters of $H$ equals the 
number of real conjugacy classes of $H$. Since $|H|$ is odd, $\alpha$  
is the trivial character of $H$. By Frobenius reciprocity, $\chi$ is 
an irreducible constituent of $\Ind_H^G (\alpha) $. On the other hand, 
since $H$ is a $p$-complement of $G$, $\Ind_H^G (\alpha)$ is the   
character of the projective indecomposable $\OG$-module corresponding 
to the trivial $kG$-module.   
\end{proof}

Combining the two results above yields the Theorem \ref{iscontained}. 
In fact we prove more.

\begin{thm} 
Let $G$ be a finite group of odd order, let $\chi \in \Irr(G)$ and let 
$(Q,\delta )$ be a Navarro vertex of $\chi$. Then there exists an 
$\OG$-lattice affording $\chi$ with vertex $Q$. In particular, $Q$ is 
contained in an anchor of $\chi $.
\end{thm}

\begin{proof}   
Let $(W, \gamma)$ be a nucleus of $\chi$ such that $Q$ is a Sylow 
$p$-subgroup of $W$ and $\delta = \Res^W_Q (\alpha) $, where 
$\gamma= \alpha \beta$, with $\alpha$ a $p'$-special  character and  
$\beta$ a $p$-special character of $W$ (cf. \cite[Sections 2,3]{N}).
By Lemma ~\ref{odddecomposition}, the trivial Brauer character is a 
constituent  of $ \beta^{\circ} $. By Lemma~\ref{ples-special}, 
there exists an $\CO W$-lattice  $X$ affording $\beta$ and with 
vertex $Q$. Let $Y$ be an $\CO W$-lattice affording $\alpha $.    
Then $V=Y \otimes X $ is an $\CO W$-lattice  affording $\gamma $. 
We claim that $V$ has vertex $ Q$. Indeed if $V$ is relatively 
$R$-projective, then every  indecomposable sumand of $Y^* \otimes V$  
is relatively $R$-projective. On the other hand, since  $\alpha$ has 
$p'$-degree, the $\CO W$-lattice $Y^* \otimes Y$ has a direct summand
isomorphic to the trivial  $\CO W$-module.   
Thus, $Y^* \otimes  Y \otimes X$ has a  direct summand isomorphic 
to $X$. Since $Q$ is a vertex of $X$, and $Q$ is a Sylow 
$p$-subgroup of $W$,  it follows that $R$ is a Sylow $p$-subgroup 
of  $W$. Then $\Ind_W ^G (V)$ is an $\OG$-lattice with character 
$\chi = \Ind_W ^G (\gamma)$. Clearly, $\Ind_W ^G (V)$ is relatively 
$\OQ$-projective.  Suppose if possible that $\Ind_W ^G (V)$ is 
relatively $\OR$-projective for some proper subgroup $R$ of $Q$, say 
$\Ind_W ^G (V)$  is a summand of $\Ind_R^G(X)$ for some $R$ 
properly contained in $Q$ and  for some  $\OR$-lattice $X$. By the 
Mackey formula, $V$ is a summand of 
$\Ind_{W \cap \,^xR  }^ W (\Res^{\,^xR}_{ W\cap \,^xR }  \, ^xX)$  
for some $x \in G$. This is a contradiction as $|\, ^x R| < |Q|$.  
Thus $Q$ is a vertex of $\Ind_W ^G (V)$  proving the first assertion. 
The second is immediate from the first and  Proposition \ref{pro:ples}.
\end{proof}

I. M. Isaacs and G. Navarro provided us with an example of a 
$p$-special character of a $p$-solvable group  none of whose 
irreducible Brauer constituents have degree prime to $p$. Proposition 
\ref{GN2pro}  can be used to prove that the  anchors  of the  
Isaacs-Navarro example, which we give below, are strictly contained in 
the Sylow $p$-subgroups of the ambient group  (so in particular, 
these characters  are not afforded by any lattice with full vertex). 
 
\begin{exa} \label{Navvertex-big}[Isaacs -Navarro]  
Let $p=5 $ and let $M$ be the semidirect product of an 
extraspecial group of order $5^3$ and of exponent $5$, acted on
faithfully by $Q_8 $ where the action is trivial on the center.  
Let $G =M \wr C_5$ be the wreath product of  $M$ by a
cyclic group of order $5$. In $G$, there is the normal subgroup
$N = M_1  \times   \cdots  \times  M_5 $,
with each  $M_i $  isomorphic to  $M$. Also, there is a cyclic 
subgroup  $C$ of order  $5$ that permutes the  $M_i $ transitively.
Note that $M_1$ has a  $5$-special character $\alpha$ of degree  $5$.

Let $\theta \in \Irr(N) $ be the product of $\alpha$
with trivial characters of $M_2$, $ M_3$, $M_4$  and $M_5$.  
Then $\theta $  has degree  $5$ and
$\chi = \theta^G $ is $5$-special of degree  $25$.

There is a Sylow $2$-subgroup $S$ of $G$ with the
form  $ Q_1 \times Q_2 \times ... \times Q_5$, where $Q_i$ is a
Sylow  $2$-subgroup of  $M_i$ and the $Q_i$ are
permuted transitively by  $C$.
Now $\theta_S  $ is the product of  $\alpha_{Q_1} $
with trivial characters on the other  $Q_i$.

Also, $\alpha_{Q_1} $  is the sum of the irreducible character of 
degree $2 $ and the three nontrivial linear characters, so there is 
no trivial constituent. It follows that $\theta_S$ has no  
$C$-invariant irreducible constituent.  The same is therefore
true about $\chi_S$. Then each $5$-Brauer irreducible constituent 
of $\chi$ has degree divisible by  $5$. 

The construction also  shows that if $x \in G \setminus N$,  then 
$\theta $ and $\,^x \theta $ have no irreducible Brauer constituents 
in common.  So by Proposition \ref{GN2pro}, the  anchors of $\chi$ are 
contained in $N$. 
\end{exa} 

In conjunction with Proposition \ref{GN1pro}, the following example
provides characters whose anchors are not contained in Navarro 
vertices.  The  construction is similar to that in the Isaacs-Navarro 
example above.
 
\begin{exa}   \label{Navvertexsmall}
Suppose that  $M =O_{p,p'} (M) $, and $ \alpha $  is an irreducible 
$p$-special character of $M$ such that   $\Res^M _{O_{p}(M)}\alpha $ 
is irreducible. Suppose further that   there exists a  nontrivial 
irreducible Brauer  character $\varphi$ of $ M$ such that 
$d_{\alpha, \varphi} \ne 0 $.   Let  $\beta$ be  the  irreducible 
character of $M$ with $O_p (M) $ in  the kernel of $\beta$ and  such 
that  $\beta^{\circ} $ equals the dual $\varphi^* $ of $\varphi $.   
Then $ \beta $ is $p'$-special.
   
Let $ G = M \wr C_p $.  In $G$, there is the normal subgroup
$N= M_1 \times \cdots \times  M_p  $
with each $M_i$ isomorphic to  $M$. Let $\tilde \alpha  \in \Irr(N)$ 
be the product  of $\alpha$ and the trivial characters of 
$M_2, \cdots, M_p $, let  $\tilde \beta \in \Irr(N) $ be the product  
of $\beta $ and the trivial characters of $M_2, \cdots, M_p $  and 
let $\tilde \varphi $ be the product of $\varphi$ with the trivial 
Brauer characters of $M_2, \cdots, M_p $.  Let 
$\chi= \Ind_N^G (\tilde \alpha \tilde \beta ) $.

Since $\tilde \alpha$ is $p$-special  and $\tilde \beta$ is 
$p'$-special, by results of \cite{G}, $\tilde \alpha  \tilde \beta $ 
is an irreducible character of $N$.  By construction, neither 
$\tilde\alpha $ nor $\tilde \beta $ is $G$-stable. Hence, also by 
general results on $p$-factorable characters, 
$\tilde\alpha\tilde \beta$ is not $G$-stable.    
Since $|G/N|=p $, it follows  that  $\chi$ is an irreducible character 
of $G$. Now, since $\tilde \beta$ is not $G$-stable, it is easy to see 
that $\chi $ is not $p$-factorable. On the other hand,  $N$ is a 
maximal  normal subgroup of  $G$. Thus 
$(N, \tilde \alpha\tilde \beta  )$ is a nucleus of $G$  in the sense 
of \cite{N}, and  the Sylow $p$-subgroups of $N$ are the  first 
components of the Navarro vertices of $\chi $.      

We have $(\tilde \alpha \tilde \beta  )^{\circ}=$
$\tilde \alpha ^{\circ}\tilde \beta ^{\circ} $, and $\tilde \varphi $ 
is an irreducible Brauer constituent  of $\tilde\alpha $ and 
$\tilde \beta^{\circ} = \tilde \varphi^{*} $. Since  
$\tilde\varphi (1)= \varphi(1) $ is relatively prime to $p$, it follows 
that  the trivial Brauer character of $N$ is a constituent of 
$(\tilde \alpha \tilde \beta )^{\circ} $.
Consequently, the trivial Brauer character of $G$ occurs as a 
constituent  of $\chi $.   Thus, by Proposition \ref{GN1pro}, the 
anchors of $\chi $ are not contained in $N$.
\end{exa}

\section{Lifting} \label{LiftingSection}

Let $G$ be a finite group and $\chi\in$ $\Irr(G)$. Let $P$ be an anchor 
of $\chi$. Then $k\tenO\OG e_\chi$ is a $G$-interior $k$-algebra. Since 
$$(k \tenO \OG e_\chi)^G = Z(k \tenO \OG e_\chi)$$
is a local $k$-algebra, it follows that $k \tenO \OG e_\chi$ is a 
primitive $G$-interior $k$-algebra. Since
$$k \tenO (\OG e_\chi)^P \subseteq (k \tenO \OG e_\chi)^P,$$
$k \tenO \OG e_\chi$ has a defect group $Q$ contained in $P$. 
We will see below that we often (but not always) have equality here.
If $\chi^\circ \in \IBr(G)$, then there is, up to isomorphism, a unique 
$\OG$-lattice $L$ affording $\chi$, and $k \tenO L$ is the unique 
simple $kG$-module with Brauer character $\chi^\circ$, up to 
isomorphism. We have seen above that in that case the $G$-interior 
$\CO$-algebra $\OG e_\chi$ is isomorphic to $\End_\CO(L)$. This
implies that the $G$-interior $k$-algebra 
$k \tenO \OG e_\chi$ is isomorphic to $\End_k(k \tenO L)$. Thus the
anchor $P$ of $\chi$ is a vertex of $L$, and the defect group $Q$ 
of $k\tenO \OG e_\chi$ is a vertex of $k \tenO L$. The examples
\ref{samelift} and \ref{biglift} below illustrate the cases where
$Q=P$ and $Q < P$, respectively.

\begin{exa} \label{samelift}
Let $G$ be the symmetric group $\mathfrak{S}_n$, for a positive integer 
$n$. Let $\chi \in \Irr(G)$ such that $\chi^\circ \in \IBr(G)$, and let 
$L$ be an $\OG$-lattice affording $\chi$. We claim that $L$ and 
$k \tenO L$ have the same vertices.

Indeed, let $\lambda$ be the partition of $n$ labelling $\chi$. Since 
the Specht lattice $S_\CO^\lambda$ is an $\OG$-lattice affording 
$\chi$, the uniqueness of $L$ implies that $S_\CO^\lambda \cong L$. 
Thus the $kG$-module $S_k^\lambda \cong k \tenO S_\CO^\lambda 
\cong k \tenO L$ has Brauer
character $\chi^\circ$ and is therefore simple.

A result by Hemmer (cf. \cite{H}) implies that $S_k^\lambda$ lifts to 
a $p$-permutation $\OG$-lattice $M$. Then $K \tenO M$ is a simple 
$KG$-module; that is, $K \tenO M \cong$ 
$S_K^\mu \cong K \tenO S_\CO^\mu$ for some partition $\mu$ of $n$. 
Moreover, $S_k^\mu$ is isomorphic to $k \tenO M \cong$ $S_k^\lambda$; 
in particular, we have
$$\Hom_{kG}(S_k^\lambda,S_k^\mu) \neq 0 \neq 
\Hom_{kG}(S_k^\mu, S_k^\lambda).$$
Suppose first that $p>2$. Then \cite[Proposition 13.17]{J} implies 
that $\lambda \ge \mu$ and $\mu \ge \lambda$, hence 
$\mu = \lambda$. The uniqueness of $L$ implies that 
$S_\CO^\lambda \cong L \cong M$; in particular, $L$ is a
$p$-permutation $\OG$-lattice. Hence $L$ and $k \tenO L$ have the same 
vertices.

It remains to consider the case $p=2$. In this case a theorem by James 
and Mathas (cf. \cite{JM}) implies that either $\lambda$
is $2$-regular, or the conjugate partition $\lambda'$ is $2$-regular, 
or $n=4$ and $\lambda = (2,2)$. The last alternative is
trivial. Multiplying by the sign character, if necessary, we may 
therefore assume that $\lambda$ is $2$-regular. If $\mu$ is
also $2$-regular then we certainly have $\lambda = \mu$. If $\mu'$ is 
$2$-regular then we have $\lambda = \mu'$, in a similar
way. Now, arguing as in the case $p>2$, we conclude that $L$ and 
$k \tenO L$ have the same vertices.
\end{exa}

\begin{exa} \label{biglift}  
Let $p=2 $, $G= GL(2, 3)$ and $N =SL(2,3) $. Let $R$ be the unique  
Sylow $2$-subgroup  of $N$ and $ H$ a complement of $R$ in $N$. Let 
$\tau$ be the  $2$-dimensional irreducible character of $R$ and let 
$\eta$ be the unique extension of $\tau$ to $N$ with determinantal 
order a power of $2$ (cf. Corollary (6.28) in \cite{I}). Let $\chi$ 
be an  extension of  $\eta $ to $G$. Then $\chi$ is $2$-special, by 
\cite[Proposition 40.5]{Hu}. Further,  $\chi^{\circ} $ is 
irreducible and  equals $\Ind_N^G (\psi )$, where $\psi$ is a linear 
Brauer character of $N$. (Note that  the restriction  of 
$\chi^{\circ} $ to $H$ is a sum of two distinct irreducible Brauer 
characters). 

Thus, $R$ is a vertex of the unique $kG$-module affording 
$\chi^{\circ}$ and $R$ is contained in some (and hence every) vertex 
of the  $\OG$-lattice  affording $\chi$. Since $\chi$ is not induced   
from any character of $N$, and $G/N$ is a $2$-group,   
Green's indecomposability theorem implies that the $\OG$-lattice 
affording  $\chi$ is not relatively $N$-projective.   
Hence $ R$ is properly contained in a vertex of the $\OG$-lattice 
affording $\chi $, which is consequently a Sylow $2$-subgroup of $G$.
\end{exa}

\begin{rem}  \label{non-endopermutation}   
Let $G$ be a finite $p$-solvable group and $\chi\in\Irr(G)$ such 
that $\chi^\circ \in \IBr(G)$. Let $L$ be an $\OG$-lattice 
affording  $\chi$. Suppose, as in the above example that a vertex 
$P$ of $L$ strictly contains a vertex $R$ of $k \tenO L$.    
Let $S$ be an $\OP$-lattice source of $L$.      
We claim that $S$ is not an endopermutation module. Indeed, assume 
the contrary. Since $P$ is a vertex of $S$ and since $S$ is 
endopermutation,  $k\tenO S$ is an indecomposable endopermutation 
$kP$-module  with vertex $P$. On the other hand, $k\tenO S$   
is a direct summand of $k\tenO L $, $k\tenO L$ has vertex $R$,  
and $R$ is strictly contained in $P$, a contradiction.
\end{rem}


\textbf{Acknowledgements.} 
The authors are grateful to Susanne Danz for her help with the example 
on symmetric groups.  They would also like to thank Gabriel Navarro for 
providing us with Lemma ~\ref{odddecomposition}  and  Marty Isaacs and 
Gabriel Navarro  for providing us with Example~\ref{Navvertex-big}. Work 
on this paper started when the second author visited the City University 
of London in 2013, with the kind support of the DFG, SPP 1388. He is 
grateful for the hospitality received at the City University.


\end{document}